\documentclass[oneside,english]{amsart}
\usepackage[T1]{fontenc}
\usepackage[latin9]{inputenc}
\usepackage{amstext}
\usepackage{amsthm}
\usepackage{amssymb}

\makeatletter
\numberwithin{equation}{section}
\numberwithin{figure}{section}
\theoremstyle{plain}
\newtheorem{thm}{\protect\theoremname}[section]
\theoremstyle{definition}
\newtheorem{defn}[thm]{\protect\definitionname}

\makeatother

\usepackage{babel}
\providecommand{\definitionname}{Definition}
\providecommand{\theoremname}{Theorem}

\begin{document}
\title{Polynomial central set theorem near zero}
\author{ANINDA CHAKRABORTY AND SAYAN GOSWAMI}
\email{anindachakraborty2@gmail.com}
\address{Government General Degree College at Chapra, Chapra, Nadia-741123,
West Bengal, India}
\email{sayan92m@gmail.com}
\address{Department of Mathematics, University of Kalyani, Kalyani, Nadia-741235,
West Bengal, India}
\begin{abstract}
{\normalsize{}N. Hindman and I. Leader introduced the set $0^{+}$
of ultrafilters on $\left(0,1\right)$ and characterize smallest ideal
of $\left(0^{+},+\right)$ and proved the Central Set Theorem near
$0$. Recently Polynomial Central Set Theorem has been proved by V.
Bergelson, J. H. Johnson Jr. and J. Moreira. In this article, we will
prove Polynomial Central Set Theorem near $0$.}{\normalsize\par}
\end{abstract}

\maketitle

\section{Introduction}

Given a discrete semigroup $\left(S,\cdot\right)$, it is well known
that one can extend the operation $\cdot$ to $\beta S$ , the Stone-\v{C}ech
compactification of $S$ so that $\left(\beta S,\cdot\right)$ is
a right topological semigroup (i.e. for each $p\in\beta S$, the function
$\rho_{p}:S\rightarrow S$, defined by $\rho_{p}\left(q\right)=q\cdot p$,
is continuous) with $S$ contained in the topological center (i.e.
for each $x\in S$, the function $\lambda_{x}:\beta S\rightarrow\beta S$,
defined by $\lambda_{x}\left(p\right)=x\cdot p$, is continuous).
Further, this operation has frequently proved to be useful in Ramsey
Theory. See \cite{key-7} for an elementary introduction to the semigroup
$\left(\beta S,\cdot\right)$ and its combinatorial applications.

It is also well known that if $S$ is not discrete, such an extension
may not be possible. (See Section 2 of \cite{key-6} where it is shown
how bad the situation is for any dense subsemigroup of $\left(\left[0,\infty\right],+\right)$.)

Surprisingly, however, it has turned out to be possible to use the
algebraic structure of $\beta\mathbb{R}_{d}$ to obtain Ramsey Theoretic
results that are stated in terms of the usual topology on $\mathbb{R}$.
(Given a topological space $X$, the notation $X_{d}$ represents
the set $X$ with the discrete topology.)

Specifically, consider the semigroup $\left(\left(0,1\right),\cdot\right)$
and let 
\[
0^{+}=\bigcap_{\epsilon>0}\mathit{cl}_{\beta\left(0,1\right)_{d}}\left(0,\epsilon\right).
\]
 Then $0^{+}$ is a two sided ideal of $\left(\beta\left(0,1\right)_{d},\cdot\right)$,
so contains the smallest ideal of $\left(\beta\left(0,1\right)_{d},\cdot\right)$
and it also a subsemigroup of $\left(\beta\mathbb{R}_{d},+\right)$.
As a compact right topological semigroup, $0^{+}$ has a smallest
two sided ideal.
\begin{defn}
\label{Definition 1.1} \cite[Definition 3.2]{key-6} Let $S$ be
a dense subsemigroup of $\left(\left(0,\infty\right),+\right)$.

(a) $K$ is the smallest ideal of $\left(0^{+},+\right)$.

(b) A subset $B$ of $S$ is syndetic near $0$ if and only if for
every $\epsilon>0$ there exist some $F\in\mathcal{P}_{f}\left(\left(0,\epsilon\right)\cap S\right)$
and some $\delta>0$ such that $S\cap\left(0,\delta\right)\subseteq\bigcup_{t\in F}\left(-t+B\right)$.
\end{defn}

The sets which are the elements of the ultrafilter in $K$ are of
great importance.
\begin{defn}
\label{Definition 1.2} \cite[Definition 3.4]{key-6} Let $S$ be
a dense subsemigroup of $\left(\left(0,\infty\right),+\right)$. A
subset $A$ of $S$ is piecewise syndetic near $0$ if and only if
there exist sequences $\left\langle F_{n}\right\rangle _{n=1}^{\infty}$
and $\left\langle \delta_{n}\right\rangle _{n=1}^{\infty}$ such that 

(1) for each $n\in\mathbb{N}$, $F_{n}\in\mathcal{P}_{f}\left(\left(0,1/n\right)\cap S\right)$
and $\delta_{n}\in\left(0,1/n\right)$ and 

(2) for all $G\in\mathcal{P}_{f}\left(S\right)$ and all $\mu>0$
there is some $x\in\left(0,\mu\right)\cap S$ such that for all $n\in\mathbb{N}$,
$\left(G\cap\left(0,\delta_{n}\right)\right)+x\subseteq\bigcup_{t\in F_{n}}\left(-t+A\right)$.
\end{defn}

\begin{thm}
\label{Theorem 1.3} \cite[Theorem 3.5]{key-6} Let $S$ be a dense
subsemigroup of $\left(\left(0,\infty\right),+\right)$ and let $A\subseteq S$.
Then $K\cap\mathit{cl}_{\beta S_{d}}A\neq\emptyset$ if and only if
$A$ is piecewise syndetic near $0$.
\end{thm}

The idempotents in $K$ are called minimal idempotents. Sets belongs
to the minimal idempotents are called ``central''.
\begin{defn}
\label{Definition 1.4} \cite[Definition 4.1]{key-6} Let $S$ be
a dense subsemigroup of $\left(\left(0,\infty\right),+\right)$.

(a) A set $A\subseteq S$ is central near $0$ if and only if there
is some idempotent $p\in K$ with $A\in p$.

(b) A family $\mathcal{A}\subseteq\mathcal{P}\left(S\right)$ is collectionwise
piecewise syndetic near $0$ if and only if there exist functions
\[
F:\mathcal{P}_{f}\left(\mathcal{A}\right)\longrightarrow\times_{n=1}^{\infty}\mathcal{P}_{f}\left(\left(0,1/n\right)\cap S\right)
\]
 and 
\[
\delta:\mathcal{P}_{f}\left(\mathcal{A}\right)\longrightarrow\times_{n=1}^{\infty}\left(0,1/n\right)
\]
 such that for every $\mu>0$, every $G\in\mathcal{P}_{f}\left(S\right)$,
and every $\mathcal{H}\in\mathcal{P}_{f}\left(\mathcal{A}\right)$,
there is some $t\in\left(0,\mu\right)\cap S$ such that for every
$n\in\mathbb{N}$ and every $\mathcal{F}\in\mathcal{P}_{f}\left(\mathcal{H}\right)$,
\[
\left(G\cap\left(0,\delta\left(\mathcal{F}\right)_{n}\right)\right)+t\subseteq\bigcup_{x\in F\left(\mathcal{F}\right)_{n}}\left(-x+\bigcap\mathcal{F}\right).
\]
\end{defn}

So every member of a family which is collectionwise piecewise syndetic
is piecewise syndetic. Here is a result of our interest:
\begin{thm}
\label{Theorem 1.5} \cite[Theorem 4.2]{key-6} Let $S$ be a dense
subsemigroup of $\left(\left(0,\infty\right),+\right)$ and let $A\in\mathcal{P}\left(S\right)$.
There exists $p\in K$ such that $\mathcal{A}\subseteq p$ if and
only if $\mathcal{A}$ is collectionwise piecewise syndetic near $0$.
\end{thm}

Now, we need to recall some definitions from \cite{key-6}.
\begin{defn}
\label{Definition 1.6} \cite[Definition 2.4]{key-6} Let $S$ be
a subsemigroup of $\left(\mathbb{R},+\right)$. 

(a) Let $\alpha:\beta S_{d}\rightarrow\left[-\infty,\infty\right]$
be the continuous extension of the identity function.

(b) $B\left(S\right)=\,\left\{ p\in\beta S_{d}:\alpha\left(s\right)\notin\left\{ -\infty,\infty\right\} \right\} $.

(c) Given $x\in\mathbb{R}$, 
\[
x^{+}=\left\{ p\in B\left(S\right):\alpha\left(p\right)=x\text{ and }\left(x,\infty\right)\cap S\in p\right\} \quad\text{and}
\]
\[
x^{-}=\left\{ p\in B\left(S\right):\alpha\left(p\right)=x\text{ and }\left(-\infty,x\right)\cap S\in p\right\} .
\]

(d) $U=\bigcup_{x\in\mathbb{R}}x^{+}$ and $D=\bigcup_{x\in\mathbb{R}}x^{-}$.
\end{defn}

The set $B\left(S\right)$ is the set of bounded ultrafilters on $S$
. That is, an ultrafilter $p\in\beta S_{d}$ is in $B\left(S\right)$
if and only if there is some $n\in\mathbb{N}$ with $\left[-n,n\right]\cap S\in p$.
The following one is the combinatorial characterization of central
set which we will use in later:
\begin{thm}
\label{Theorem 1.7} \cite[Theorem 4.7]{key-6} Let $S$ be a dense
subsemigroup of $\left(\left(0,\infty\right),+\right)$ and let $A\subseteq S$
. Statements (a) and (b) are equivalent.

(a) $A$ is central near $0$.

(b) There is a decreasing sequence $\left\langle C_{n}\right\rangle _{n=1}^{\infty}$
subsets of $A$ such that 
\begin{enumerate}
\item for all $n\in\mathbb{N}$ and all $x\in C_{n}$, there is some $m\in\mathbb{N}$
with $C_{m}\subseteq-x+C_{n}$ and 
\item $\left\{ C_{n}:n\in\mathbb{N}\right\} $ is collectionwise piecewise
syndetic near $0$.
\end{enumerate}
\end{thm}

\begin{defn}
\label{Definition 1.8} \cite[Definition 4.9]{key-6} Let $S$ be
a countable dense subsemigroup of $\left(\left(0,\infty\right),+\right)$.

(a) $\Phi=\left\{ f|f:\mathbb{N}\rightarrow\mathbb{N}\text{ and for all }n\in\mathbb{N},\,f\left(n\right)\leq n\right\} $. 

(b) 
\begin{align*}
\mathcal{Y} & =\left\{ \left\langle \left\langle y_{i,t}\right\rangle _{t=1}^{\infty}\right\rangle _{i=1}^{\infty}|\text{ for each }i\in\mathbb{N},\right.\\
 & \left.\left\langle y_{i,t}\right\rangle _{t=1}^{\infty}\text{ is a sequence in }S\cup-S\cup\left\{ 0\right\} \text{ and }\sum_{t=1}^{\infty}\left|y_{i,t}\right|\text{ converges}\right\} .
\end{align*}
\end{defn}

\begin{thm}
\label{Theorem 1.9}\cite[Theorem 4.11]{key-6} Let $S$ be a countable
dense subsemigroup of $\left(\left(0,\infty\right),+\right)$ and
let $A$ be central near $0$ in $S$. Let $Y=\left\langle \left\langle y_{i,t}\right\rangle _{t=1}^{\infty}\right\rangle _{i=1}^{\infty}\in\mathcal{Y}$.
Then there exist sequences $\left\langle a_{n}\right\rangle _{n=1}^{\infty}$
in $S$ and $\left\langle H_{n}\right\rangle _{n=1}^{\infty}$ in
$\mathcal{P}_{f}\left(\mathbb{N}\right)$ such that

(a) for each $n\in\mathbb{N}$, $a_{n}<1/n$ and $\max H_{n}<\min H_{n+1}$
and

(b) for each $f\in\Phi$, $FS\left(\left\langle a_{n}+\sum_{t\in H_{n}}y_{f\left(n\right),t}\right\rangle _{n=1}^{\infty}\right)\subseteq A$.
\end{thm}

Given an infinite set $X$, we denote by $\mathcal{P}_{f}\left(X\right)$
the family of all finite non-empty subsets of $X$, i.e., $\mathcal{F}\left(X\right)=\left\{ \alpha\subseteq X:0<\left|\alpha\right|<\infty\right\} $.
Let $G$ be a countable commutative semigroup and let $\left\langle x_{n}\right\rangle _{n\in\mathbb{N}}$
be an injective sequence in $G$. For each $\alpha\in\mathcal{P}_{f}\left(\mathbb{N}\right)$
define $x_{\alpha}=\sum_{n\in\alpha}x_{n}$. The IP-set generated
by $\left\langle x_{n}\right\rangle _{n\in\mathbb{N}}$ is the set
$FS\left\langle x_{n}\right\rangle =\left\{ x_{\alpha}:\alpha\in\mathcal{P}_{f}\left(\mathbb{N}\right)\right\} $.
Clearly $x_{\alpha\cup\beta}=x_{\alpha}+x_{\beta}$ for any disjoint
$\alpha,\beta\in\mathcal{P}_{f}\left(\mathbb{N}\right)$. Moreover,
if $\left\langle y_{\alpha}\right\rangle _{\alpha\in\mathcal{P}_{f}\left(\mathbb{N}\right)}$
is any sequence indexed by $\mathcal{P}\left(\mathbb{N}\right)$ such
that $x_{\alpha\cup\beta}=x_{\alpha}+x_{\beta}$ for any disjoint
$\alpha,\beta\in\mathcal{P}_{f}\left(\mathbb{N}\right)$, then the
set $\left\{ y_{\alpha}:\alpha\in\mathcal{P}_{f}\left(\mathbb{N}\right)\right\} $
is an IP-set (generated by $\left\langle y_{\left\{ n\right\} }\right\rangle _{n\in\mathbb{N}}$).
For this reason we will denote IP-sets by $\left\langle y_{\alpha}\right\rangle _{\alpha\in\mathcal{P}_{f}\left(\mathbb{N}\right)}$,
with the understanding that they are generated by the singletons $y_{n},n\in\mathbb{N}$.
\begin{defn}
\label{Definition 1.10} Let $\left\langle x_{\alpha}\right\rangle _{\alpha\in\mathcal{P}_{f}\left(\mathbb{N}\right)},\left\langle y_{\alpha}\right\rangle _{\alpha\in\mathcal{P}_{f}\left(\mathbb{N}\right)}$
be IP-sets in a countable commutative semigroup $G$. 
\begin{enumerate}
\item We define a partial order on $\mathcal{P}_{f}\left(\mathbb{N}\right)$
by letting $\alpha<\beta$ whenever $\alpha,\beta\in\mathcal{P}_{f}\left(\mathbb{N}\right)$
and $\max_{i\in\alpha}i<\min_{j\in\beta}j$.
\item We say that $\left\langle x_{\alpha}\right\rangle _{\alpha\in\mathcal{P}_{f}\left(\mathbb{N}\right)}$
is a sub-IP-set of $\left\langle y_{\alpha}\right\rangle _{\alpha\in\mathcal{P}_{f}\left(\mathbb{N}\right)}$
if there exist $\alpha_{1}<\alpha_{2}<\cdots$ in $\mathcal{P}_{f}\left(\mathbb{N}\right)$
such that $x_{n}=y_{\alpha_{n}}$ for all $n\in\mathbb{N}$.
\end{enumerate}
The following theorem is important:
\end{defn}

\begin{thm}
\label{Theorem 1.11} \cite[Theorem 3.1]{key-6} Let $S$ be a dense
subsemigroup of $\left(\left(0,\infty\right),+\right)$ and let $A\subseteq S$.
There exists $p=p+p$ in $0^{+}$ with $A\in p$ iff there is some
sequence $\langle x_{n}\rangle_{n=1}^{\infty}$ in $S$ such that
$\sum_{n=1}^{\infty}x_{n}$ converges and $FS\left(\langle x_{n}\rangle_{n=1}^{\infty}\right)\subseteq A$
\end{thm}

\section{Proof of our results}

In \cite{key-6} the authors proved van der Waerden's Theorem near
$0$ and Central Set Theorem near $0$. Polynomial van der Waerden's
Theorem proved in \cite{key-1.1}, \cite{key-9} and Polynomial van
der Waerden's Theorem near $0$ proved in \cite{key-3}. Traditionally,
van der Waerden\textquoteright s Theorem can be proved from Hales-Jewett
Theorem which is treated as combinatorial proof. 

Conventionally $\left[t\right]$ denotes the set $\left\{ 1,2,\ldots,t\right\} $
and words of length $N$ over the alphabet $\left[t\right]$ are the
elements of $\left[t\right]^{N}$ . A variable word is a word over
$\left[t\right]\cup\left\{ *\right\} $ where $*$ denotes the variable.
A combinatorial line is denoted by $L_{\tau}=\left\{ \tau\left(1\right),\tau\left(2\right),\ldots,\tau\left(t\right)\right\} $
where $\tau$ is a variable word and $L_{\tau}$ is obtained by replacing
the variable $*$ by $1,2,\ldots,t$. The following theorem is due
to Hales-Jewett \cite{key-5}.
\begin{thm}
\label{Theorem 2.1} \cite{key-5} For all values $t,r\in\mathbb{N}$
there exists a number $HJ\left(r,t\right)$ such that, if $N\geq HJ\left(r,t\right)$
and $\left[t\right]^{N}$ is $r$ colored then there exists a monochromatic
combinatorial line.
\end{thm}

In \cite{key-2}, authors proved Polynomial Hales-Jewett Theorem and
in \cite{key-9} it is proved combinatorially. 

The first Central Set Theorem was proved in \cite{key-4}. In \cite{key-1.1},
authors proved the Polynomial Central Set Theorem using IP van der
Waerden's Theorem. Before stating those results we need this definition:
\begin{defn}
\label{Definition 2.2} Given a map $f:H\rightarrow G$ between countable
commutative groups we say that $f$ is a polynomial map of degree
$0$ if it is constant. We say that $f$ is a polynomial map of degree
$d\in\mathbb{N}$, if it is not a polynomial map of degree $d-1$
and for every $h\in H$, the map $x\longrightarrow f\left(x+h\right)-f\left(x\right)$
is a polynomial of degree $\leq d-1$. Finally we denote by $\mathbb{P}\left(G,H\right)$
the set of all polynomial maps $f:G\rightarrow H$ with $f\left(0\right)=0$. 
\end{defn}

Note that homomorphisms are elements of $\mathbb{P}\left(G,H\right)$
having degree $1$.
\begin{thm}
\label{Theorem 2.3} Let $j\in\mathbb{N}$, let $G$ be a countable
abelian group and let $F$ be a finite family of polynomial maps from
$G^{j}$ to $G$ such that $f\left(0\right)=0$ for each $f\in F$.
Then for every piecewise syndetic set $A\subseteq G$ and every IP
set $\left\langle y_{\alpha}\right\rangle _{\alpha\in\mathcal{P}_{f}\left(\mathbb{N}\right)}$
in $G^{j}$ there exist $a\in A$ and $\alpha\in\mathcal{P}_{f}\left(\mathbb{N}\right)$
such that $a+f\left(y_{\alpha}\right)\in A$ for every $f\in F$.
\end{thm}

\begin{thm}
\label{(IP polynomial van der Waerden theorem for abelian groups)}
\cite[Theorem 2.10]{key-1} Let $G,H$ be countable abelian groups
and let $F\subseteq\mathbb{P}\left(H,G\right)$ be a finite subset.
Then for every finite partition $G=C_{1}\cup\ldots\cup C_{r}$ and
every IP set $\left\langle y_{\alpha}\right\rangle _{\alpha\in F}$
in $H$ there exist $i\in\left\{ 1,\ldots,r\right\} $, $a\in C_{i}$
and $\alpha\in\mathcal{P}_{f}\left(\mathbb{N}\right)$ such that $a+f\left(y_{\alpha}\right)\in C_{i}$
for every $f\in\mathcal{P}_{f}\left(\mathbb{N}\right)$.
\end{thm}

\begin{thm}
\label{(Multidimensional polynomial central sets theorem)} \cite[Theorem 3.8]{key-1}
Let $G$ be a countable commutative semigroup, let $j\in\mathbb{N}$,
and let $\left\langle y_{\alpha}\right\rangle _{\alpha\in\mathcal{F}}$
be an IP-set in $G^{j}$ . Let $F\subseteq\mathbb{P}\left(G^{j},G\right)$
and let $A\subseteq G$ be a central set (or, if $G=\mathbb{Z}^{n}$,
let $A$ be a $D$-set). Then there exist an IP-set $\left\langle x_{\beta}\right\rangle _{\beta\in\mathcal{F}}$
in $G$ and a sub-IP-set $\left\langle z_{\beta}\right\rangle _{\beta\in\mathcal{F}}$
of $\left\langle y_{\alpha}\right\rangle _{\alpha\in F}$ such that
\[
\forall f\in F\qquad\forall\beta\in\mathcal{F}\qquad x_{\beta}+f\left(z_{\beta}\right)\in A.
\]
\end{thm}

Let $\left(S,+\right)$ be a dense subsemigroup of $\left(\mathbb{R},+\right)$
containing $0$ such that $\left(S\cap\left(0,1\right),\cdot\right)$
is a subsemigroup of $\left(\left(0,1\right),\cdot\right)$. Also
let $\mathcal{P}$ is the set of all polynomial defined on $\mathbb{R}$.
For any set $X$, $\mathcal{P}_{f}\left(X\right)$ denotes the set
of all finite subset of $X$. Let $\left\langle S_{i}\right\rangle _{i=1}^{\infty}$
be a sequence and $FS\left(\left\langle S_{i}\right\rangle _{i=1}^{\infty}\right)$
is an IP set in $S\cap\left(0,1\right)$ and therefore convergent
and consequently $\lim S_{n}=0$ \cite[Theorem 3.1]{key-6}.
\begin{thm}
\label{Theorem 2.6} Let $\left(S,+\right)$ be a dense subsemigroup
of $\left(\mathbb{R},+\right)$ containing $0$ such that $\left(S\cap\left(0,1\right),\cdot\right)$
is a subsemigroup of $\left(\left(0,1\right),\cdot\right)$. Let $F$
be a set of finite polynomials each of which vanishes at $0$. Then
for any $r\in\mathbb{N}$ and $\epsilon>0$, if $S\cap\left(0,\epsilon\right)$
is $r$-colored and for any IP set $FS\left(\left\langle S_{i}\right\rangle _{i=1}^{\infty}\right)$
then one of the cells will contain a configuration of the form 
\[
\left\{ a+p\left(s_{n}\right):\text{ for some }\ a\in S,\,n\in\mathbb{N}\right\} .
\]
\end{thm}

\begin{proof}
Let $\left[q\right]$ be the set of co-efficients of the polynomials
in $F$ and $N\left(q,r,d\right)$ exists such tat for any $r$-coloring
of 
\[
Q=\left[q\right]^{N}\times\left[q\right]^{N\times N}\times\ldots\times\left[q\right]^{N^{d}}
\]
and $\gamma\in\mathcal{P}_{f}\left(\mathbb{N}\right)$, one has the
following configuration:
\[
\left\{ a\oplus x_{1}\gamma\oplus x_{2}\left(\gamma\times\gamma\right)\oplus\ldots\oplus x_{d}\gamma^{d}:1\leq x_{i}\leq q\right\} .
\]
 Let us choose $\left\{ C_{1},C_{2},\ldots,C_{q}\right\} $ from the
base sequence $S_{1}>S_{2}>S_{3}\ldots$ satisfying $\max B<\epsilon$,
where
\begin{align*}
B & =\left\{ \sum_{i=1}^{N}a_{i}c_{i}+\sum_{i_{1},i_{2}=1}^{N}a_{i_{1}i_{2}}c_{i_{1}}c_{i_{2}}+\ldots+\right.\\
 & \left.\sum_{i_{1},i_{2},\ldots,i_{d}=1}^{N}a_{i_{1}i_{2}\ldots i_{d}}c_{i_{1}}c_{i_{2}}\ldots c_{i_{d}}:\left(a_{1},\ldots,a_{N},a_{11},\ldots,a_{NN},\ldots a_{NN\ldots N}\right)\in Q\right\} 
\end{align*}
Now, 
\begin{align*}
C\left(a_{1},\ldots,a_{N},a_{11},\ldots,a_{NN},\ldots a_{NN\ldots N}\right) & =\sum_{i=1}^{N}a_{i}c_{i}\\
 & +\sum_{i_{1},i_{2}=1}^{N}a_{i_{1}i_{2}}c_{i_{1}}c_{i_{2}}+\ldots\\
 & +\sum_{i_{1},i_{2},\ldots,i_{d}=1}^{N}a_{i_{1}i_{2}\ldots i_{d}}c_{i_{1}}c_{i_{2}}\cdots c_{i_{d}}.
\end{align*}

Now any coloring of $S\cap\left(0,\epsilon\right)$ will induce a
$r$-coloring of $B$ by the map $C$ and then a $r$-coloring of
$Q$.

Then from polynomial Hales-Jewett Theorem the configuration 
\[
\left\{ a\oplus x_{1}\gamma\oplus x_{2}\left(\gamma\times\gamma\right)\oplus\ldots\oplus x_{d}\gamma^{d}:1\leq x_{i}\leq q\right\} 
\]
 is monochromatic. Now, 
\begin{align*}
C\left(a\oplus x_{1}\gamma\oplus x_{2}\left(\gamma\times\gamma\right)\oplus\ldots\oplus x_{d}\gamma^{d}\right) & =b+x_{1}\sum_{i\in\gamma}C_{i}+\\
 & x_{2}\left(\sum_{i\in\gamma}C_{i}\right)^{2}+\ldots+x_{d}\left(\sum_{i\in\gamma}C_{i}\right)^{d},\\
\end{align*}
 where $b\in\left(0,\epsilon\right)$. Now, one can add the zero polynomial
in $F$.

So, $\exists\,b\in\left(0,\epsilon\right)\cap S$ such that $\left\{ b+P\left(S_{\gamma}\right)\right\} $
is monochromatic.
\end{proof}
Above we actually proved this general version of the Theorem \ref{Theorem 2.6}:
\begin{thm}
Let $\left(S,+\right)$ be a dense subsemigroup of $\left(\mathbb{R},+\right)$
containing $0$ such that $\left(S\cap\left(0,1\right),\cdot\right)$
is a subsemigroup of $\left(\left(0,1\right),\cdot\right)$. Let $F$
be a set of finite polynomials each of which vanishes at $0$. Then
for any $r\in\mathbb{N}$, $\epsilon>0$, $\exists$ a finite set
$B\subseteq\left(0,\epsilon\right)$ such that for any $r$-coloring
of $B\cap S$ and for any IP set $FS\left(\left\langle S_{i}\right\rangle _{i=1}^{\infty}\right)$,
$\exists$ $b\in S$ and $\gamma\subseteq\mathcal{P}_{f}\left(\mathbb{N}\right)$
such that $\left\{ b+P\left(S_{\gamma}\right):S_{\gamma}=\sum_{i\in\gamma}C_{i}\right\} $
is monochromatic.
\end{thm}

Let $G=B$ in Definition \ref{Definition 1.2} where $r=\left|\delta_{1}\right|$,
then $\exists\,x\in\left(0,\mu\right)\cap S$ such that, $\left(G\cap\left(0,\delta_{1}\right)\right)+x\subseteq\bigcup_{t\in F_{1}}\left(-t+A\right)$
and this shows $A$ contains a polynomial progression of the form
$b+a+P\left(S_{n}\right)\,\forall P\in F$ for some $S_{n}$ from
a given IP set. So, we conclude:
\begin{thm}
Let $\left(S,+\right)$ be a dense subsemigroup of $\left(\mathbb{R},+\right)$
containing $0$ such that $\left(S\cap\left(0,1\right),\cdot\right)$
is a subsemigroup of $\left(\left(0,1\right),\cdot\right)$. Let $F$
be any given system of finite polynomials each of which vanishes at
$0$. Then for any piecewise syndetic $A\subseteq S$, $\epsilon>0$
and for any IP set $FS\left(\left\langle S_{i}\right\rangle _{i=1}^{\infty}\right)$,
$\exists$ $a\in S\cap\left(0,\epsilon\right)$ and $n\subseteq\mathbb{N}$
such that $\left\{ a+P\left(S_{n}\right):P\in F\right\} \subseteq A$.
\end{thm}

Now using a combinatorial characterization of central set \ref{Theorem 1.7},
we will deduce Polynomial Central Set Theorem for arbitrary semigroup
$\left(S,+\right)$.
\begin{thm}
Let $\left(S,+\right)$ be a dense subsemigroup of $\left(\mathbb{R},+\right)$
containing $0$ such that $\left(S\cap\left(0,1\right),\cdot\right)$
is a subsemigroup of $\left(\left(0,1\right),\cdot\right)$. Let $F$
be any given system of finite polynomials each of which vanishes at
$0$. Then for any central set $A\subseteq S$, $\epsilon>0$ and
for any IP sequence $\left\langle S_{i}\right\rangle _{i=1}^{\infty}$,
there exist an IP sequence $\left\langle H_{n}\right\rangle _{n=1}^{\infty}$in
$\left(0,\epsilon\right)$ and a sub IP sequence $\left\langle C_{i}\right\rangle _{i=1}^{\infty}$
of $\left\langle S_{i}\right\rangle _{i=1}^{\infty}$ such that 
\[
\left\{ H_{\alpha}+P\left(C_{\alpha}\right):P\in F,\alpha\in\mathcal{P}_{f}\left(\mathbb{N}\right)\right\} \subseteq A.
\]
\end{thm}

\begin{proof}
As $A$ is central, thus there exists a chain of piecewise syndetic
sets 
\[
A\supseteq A_{1}\supseteq A_{2}\supseteq A_{3}\supseteq\ldots
\]
 satisfying theorem \ref{Theorem 1.7}.

Now as $A_{1}$ is piecewise syndetic, then for $\epsilon=1,\ \exists\,a_{1}\in S\cap\left(0,1\right)$
and $n_{1}\in\mathbb{N}$ such that $\left\{ a_{1}+P\left(S_{n_{1}}\right):P\in F\right\} \subseteq A_{1}$.
Now, 
\[
\bigcap_{P\in F}-\overline{a_{1}+P\left(S_{n_{1}}\right)}+A_{1}\supseteq A_{m_{1}}\text{ for some }m_{1}\in\mathbb{N}.
\]

So, let $F_{1}=\left\{ P,T_{S_{n_{1}}}P:P\in F\right\} $, where $T_{S_{n_{1}}}P$
is a polynomial defined by $T_{S_{n_{1}}}P\left(m\right)=P\left(m+S_{n_{1}}\right)$,
$m\in\mathbb{N}\setminus S_{n_{1}}$.

Again as $A_{m_{1}}$ is piecewise syndetic and $\epsilon=\frac{1}{2},\ \exists\,a_{2}\in S\cap\left(0,\frac{1}{2}\right)$
and $n_{2}\in\mathbb{N}$ such that $\left\{ a_{2}+P\left(S_{n_{2}}\right):P\in F_{1}\right\} \subseteq A_{m_{1}}$
which implies 

\[
\left\{ a_{2}+P\left(S_{n_{2}}\right):P\in F\right\} \subseteq A_{m_{1}}\subseteq A
\]
 and 
\[
a_{2}+T_{S_{n_{1}}}P\left(S_{n_{2}}\right)\subseteq A_{m_{1}}\subseteq\bigcap_{P\in F_{1}}\overline{-a_{1}+P\left(S_{n_{1}}\right)}+A_{1}
\]
for $p\in F$

Therefore, 
\[
\left(a_{1}+a_{2}\right)+P\left(S_{n_{1}}+S_{n_{2}}\right)\subseteq A_{1}\subseteq A
\]

which implies $a_{\alpha}+P\left(S_{n_{\alpha}}\right)\subseteq A$
for $\alpha\in\mathcal{F}\left(\left\{ 1,2\right\} \right)$.

Now, 
\begin{align*}
\left(\bigcap_{P\in F_{1}}-\overline{a_{1}+P\left(S_{n_{1}}\right)}+A_{m_{1}}\right)\cap\left(\bigcap_{P\in F}-\overline{a_{2}+P\left(S_{n_{2}}\right)}+A_{m_{1}}\right)\cap & \text{ }\\
\left(\bigcap_{P\in F}-\overline{\left(a_{1}+a_{2}\right)+P\left(S_{n_{1}}+s_{n_{2}}\right)}+A_{m_{1}}\right)\supseteq A_{m_{2}}
\end{align*}
for some $m_{2}\in\mathbb{N}.$

Where, $F_{2}=\left\{ P,T_{S_{n_{1}}}P,T_{S_{n_{2}}}P,T_{s_{n_{1}}+S_{n_{2}}}P:P\in F\right\} $
and as $A_{m_{2}}$ is piecewise syndetic, then for $\epsilon=\frac{1}{2^{2}},\text{ there exist }a_{3}\in S\cap\left(0,\frac{1}{2^{2}}\right)$
and $n_{3}\in\mathbb{N}$ such that $\left\{ a_{3}+P\left(S_{n_{3}}\right):P\in F_{2}\right\} \subseteq A_{m_{2}}$
which implies 
\[
\left\{ a_{\alpha}+P\left(S_{n_{\alpha}}\right):P\in F\right\} \subseteq A_{m_{2}}\subseteq A
\]
where $\alpha\in\mathcal{F}\left(\left\{ 1,2,3\right\} \right)$.

And iteratively the proof follows.
\end{proof}

\end{document}